\documentclass[12pt]{article}
\usepackage[a4paper,margin=1in]{geometry}
\usepackage{setspace}
\usepackage{makecell}
\usepackage[demo]{graphicx}
\usepackage{blindtext}
\usepackage{float}
\usepackage{titlesec}
\usepackage{booktabs}
\usepackage{enumerate,enumitem}
\usepackage{amsmath,amsthm,amssymb}
\usepackage{lipsum,comment,multirow}
\usepackage{tikz}
\usepackage{caption}
\usepackage{subcaption}
\usepackage{multirow}
\usepackage{float}
\usepackage[most]{tcolorbox}
\usepackage[
backend=biber,
style=numeric,
sorting=ynt
]{biblatex}
\usepackage{hyperref}
\hypersetup{
    colorlinks=true,
    linkcolor=blue,
    filecolor=magent,  
    citecolor=red,
}
\usepackage{latexsym,booktabs}
\usepackage{graphicx}
\graphicspath{ {figures/} }
\usepackage{algorithm}
\usepackage{algpseudocode}
\usepackage{adjustbox}
\usepackage{xcolor}
\usepackage{listings}
\newtheorem{theorem}{Theorem}[section]
\newtheorem{lemma}[theorem]{Lemma}

\newtheorem{corollary}[theorem]{Corollary}
\newtheorem{proposition}{Proposition}
\newtheorem{example}[theorem]{Example}

\def\<{\langle}
\def\>{\rangle} 
\usepackage{tikz}
\usetikzlibrary{arrows.meta, positioning}

\title{SDP Approach to Quadratic Vertex-Disjoint Paths Problem\thanks{Research supported by the Air Force Office of Scientific Research under award number FA9550-23-1-0508.}}
\author{Mingming Xu\footnote{School of Mathematical and Statistical Sciences, Clemson University, Clemson, SC,  USA. \url{mingmix@g.clemson.edu}}, Hao Hu\footnote{School of Mathematical and Statistical Sciences, Clemson University, Clemson, SC,  USA. \url{hhu2@clemson.edu}}}
 
\date{\today}
\begin{document}

\maketitle

\begin{abstract}
We study the quadratic $k$-vertex-disjoint paths problem (Q-$k$-VDP), which seeks $k$ vertex-disjoint paths in a directed graph that minimize a nonconvex quadratic objective function. We formulate the problem as a binary quadratic program and apply a systematic graph reduction to manage its dimensionality. To obtain a tractable bounding model, we drop the subtour-elimination constraints and derive a semidefinite programming (SDP) relaxation. We then solve this relaxed model within a branch-and-bound framework, where the bounds are computed from the SDP relaxation using a tailored alternating direction method of multipliers. Computational results show that our proposed method consistently outperforms Gurobi by solving more instances to optimality, especially on challenging large-scale instances.
\end{abstract}

\section{Introduction}\label{sec1}
The \emph{$k$-vertex-disjoint paths problem} ($k$-VDP) is a fundamental problem in graph theory and combinatorial optimization, with numerous practical applications in network design and routing. Formally, given a directed graph $\mathcal{G} = (\mathcal{N}, \mathcal{E})$ and a set of $k$ distinct source--target pairs $(s_1, t_1), \dots, (s_k, t_k)$, the goal is to determine whether there exist $k$ paths such that each path connects $s_i$ to $t_i$ and the paths are pairwise vertex-disjoint.

This problem arises when multiple connections must be established in limited space while remaining disjoint. It is critical in VLSI design, layout, and routing \cite{aggarwal2000node, behjat2005fast, behjat2006integer,held2011combinatorial, wu2011grip}, where physical wires must connect components on a chip without overlapping. It also appears in communication networks \cite{ogier2002distributed,AnandModiano2005}, such as establishing separate communication channels to ensure that if one path fails, the others remain active. 

If $k$ is part of the input, the $k$-VDP is NP-hard for both undirected and directed graphs \cite{karp1975computational}. For undirected graphs, the problem can be solved in polynomial time when $k=2$ \cite{seymour1980disjoint,shiloach1980polynomial,ohtsuki1981two,thomassen19802}, and more generally, when $k$ is a fixed constant \cite{robertson1995graph,kawarabayashi2012disjoint}. For directed graphs, the problem is significantly harder: it remains NP-hard even for $k=2$ \cite{fortune1980directed}. However, if the directed graph is planar, the $k$-VDP can be solved in polynomial time for any fixed $k$ \cite{schrijver1994finding}.

In many practical applications, it is natural to seek the cheapest routes, which leads to the \emph{shortest $k$-vertex-disjoint paths problem}. When $k$ is part of the input, this problem is NP-hard for both undirected and directed graphs \cite{berczi2017directed}. For undirected graphs, when $k=2$, a polynomial-time algorithm was initially developed for graphs with strictly positive edge lengths \cite{eilam1998disjoint}, and later extended to allow for nonnegative edge lengths \cite{kobayashi2019two}. Building on these foundations, it was eventually shown that the problem can be solved in polynomial time on undirected graphs for any fixed $k$ \cite{lochet2021polynomial}. For directed graphs, the problem presents different challenges. Finding two shortest vertex-disjoint paths is solvable in polynomial time if all arc lengths are positive and each directed cycle has a positive length \cite{berczi2017directed}. In the specific case of planar directed graphs, the problem becomes polynomially solvable for fixed $k$ \cite{berczi2017directed}.

In this paper, we address a quadratic variant of the $k$-VDP, which involves finding $k$ vertex-disjoint paths in a directed graph that minimize a quadratic cost. Such a cost structure arises naturally in applications where interactions between paths are significant. For instance, two paths may incur additional costs when competing for nearby resources, which can lead to congestion, interference, or operational conflicts. Conversely, they may yield cost reductions when routed compatibly through shared infrastructure. In VLSI global routing, for example, coupling between nearby wires increases routing costs, and excessive bends are penalized due to reliability concerns \cite{weste1993principles,rabaey2002digital}. Because these interactions can be either penalizing or beneficial, we focus on the nonconvex setting in which the quadratic term is generally indefinite. We formally define this variant as the \emph{quadratic $k$-vertex-disjoint paths problem} (Q-$k$-VDP).

Our main contributions are as follows. We formulate the Q-$k$-VDP as a binary quadratic program (BQP). For the planar case, we also develop a polynomial-time graph reduction procedure that significantly reduces the size of this formulation. Since the formulation contains exponentially many subtour-elimination constraints, we drop these constraints to obtain a computationally manageable relaxation, termed the subtour-relaxed BQP. This subtour-relaxed BQP still yields useful bounds for the original BQP. However, it remains difficult to solve using general-purpose solvers.

Therefore, we derive a semidefinite programming (SDP) relaxation of the subtour-relaxed BQP and apply facial reduction to reduce its dimension. Subsequently, we develop a tailored alternating direction method of multipliers (ADMM) algorithm for this relaxation and embed it within a branch-and-bound (B\&B) framework to solve the subtour-relaxed BQP to optimality.

We benchmark the performance of our approach against Gurobi, as this commercial solver natively supports mixed-integer quadratic programs and performs strongly in practice \cite{gurobi}. Experimental results show that for larger instances, our approach consistently achieves tighter optimality gaps and solves significantly more instances to optimality than Gurobi within a one-hour time limit.
 
The remainder of this paper is organized as follows. Section \ref{sec:background} reviews background on graph theory and SDP. Section \ref{sec:bqp} presents the BQP based on a multi-commodity flow representation. Section \ref{sec:dedu} introduces dimension-reduction techniques for the BQP by identifying and removing fixed arcs. Section \ref{sec:sdp_relaxation} derives the SDP relaxation of the subtour-relaxed BQP and applies facial reduction. Section \ref{sec:admm} describes the proposed ADMM algorithm. Finally, Section \ref{sec:experiments} presents the benchmark instances and computational results.

\section{Background}\label{sec:background}
In this section, we introduce the graph notation and definitions used throughout the paper. We also review basic notation for SDP and briefly summarize the essentials of facial reduction.
\subsection{Graphs}
Let $\mathcal{G}=(\mathcal{N},\mathcal{E})$ be a directed graph. An arc $e=(u,v) \in \mathcal{E}$ is outgoing from $u$ and incoming to $v$. Here, $u$ is the tail of the arc and $v$ is the head. For $v \in \mathcal{N}$, we denote its sets of outgoing, incoming, and incident arcs by $\delta_{+}(v)$, $\delta_{-}(v)$, and $\delta(v)$, respectively. For a subset of vertices $S \subseteq \mathcal{N}$, let $\mathcal{E}(S)$ be the set of arcs in $\mathcal{G}$ whose endpoints are both in $S$. A sequence of vertices $v_0, \dots, v_l$ such that $(v_{i-1}, v_i) \in \mathcal{E}$ for each $i$ is a \textit{walk}. It is a \textit{path} if all vertices are distinct, and a \textit{cycle} if $v_0 = v_l$ and all other vertices are distinct.

Let $\mathcal{G}_1, \dots, \mathcal{G}_k$ be a collection of graphs, where each shares no common vertices or arcs with any other $\mathcal{G}_i$. Their disjoint union $\bigsqcup_{i=1}^k \mathcal{G}_i$ is defined by the union of their respective vertex and arc sets:
\[
\bigsqcup_{i=1}^k \mathcal{G}_i := \left( \bigsqcup_{i=1}^k \mathcal{N}_i, \bigsqcup_{i=1}^k \mathcal{E}_i \right).
\]
If all $\mathcal{G}_i$ are isomorphic to a graph $\mathcal{G}$, their disjoint union is denoted by $k\mathcal{G}$.
\subsection{Semidefinite Programming}

Let $\mathbb{S}^n$ be the vector space of $n \times n$ real symmetric matrices, and $\mathbb{S}^n_+ = \{X \in \mathbb{S}^n : X \succeq 0\}$ be the positive semidefinite cone. We use the trace inner product $\langle U, V \rangle = \operatorname{tr}(UV)$ for $U, V \in \mathbb{S}^n$ and the standard Euclidean inner product $\langle x,y\rangle = x^\top y$ for $x, y \in \mathbb{R}^m$. Given $C, A_1, \dots, A_m \in \mathbb{S}^n$ and $b \in \mathbb{R}^m$, define the linear map $\mathbf{A} : \mathbb{S}^n \to \mathbb{R}^m$ by $\mathbf{A}(X) := (\langle A_1, X \rangle, \dots, \langle A_m, X \rangle)^\top$ and its adjoint $\mathbf{A}^*(y) := \sum_{i=1}^m y_i A_i$. The standard primal \textbf{(P)} and dual \textbf{(D)} SDP pair is:
\begin{equation*}
\begin{aligned}
\textbf{(P)} \quad \min \quad & \langle C, X \rangle \quad & \textbf{(D)} \quad \max \quad & \langle b, y \rangle \\
\text{s.t.} \quad & \mathbf{A}(X) = b, & \text{s.t.} \quad & \mathbf{A}^*(y) + S = C, \\
& X \in \mathbb{S}^n_+. & & S \in \mathbb{S}^n_+.
\end{aligned}
\end{equation*}
For a feasible \textbf{(P)}, Slater's condition holds if there exists $X \in \operatorname{int}(\mathbb{S}^n_+)$ such that $\mathbf{A}(X) = b$. Otherwise, the feasible set lies in a proper face of $\mathbb{S}^n_+$, identifiable via \cite{drusvyatskiy2017many}: 
\begin{theorem}[Theorem of the alternative]
Either \textbf{(P)} is strictly feasible, or the auxiliary system $0 \neq \mathbf{A}^*(y) \succeq 0$ with $\langle b,y\rangle \le 0$ is feasible. If \textbf{(P)} is feasible, this auxiliary system requires $\langle b,y\rangle = 0$. In this case, the primal feasible set is contained within the proper face:
\[
\mathcal{F} = (\mathbf{A}^*(y))^\perp \cap \mathbb{S}^n_+ = \{X \in \mathbb{S}^n_+ : \langle \mathbf{A}^*(y), X \rangle = 0\}.
\] 
The vector $\mathbf{A}^*(y)$ is then called an exposing vector of $\mathcal{F}$.
\end{theorem}

\section{Binary Quadratic Program}\label{sec:bqp}
The Q-$k$-VDP problem can be viewed as a special case of a quadratic multi-commodity flow problem with unit-capacity vertex constraints. To formulate it, we begin by constructing $k$ isomorphic copies of the base graph $\mathcal{G}$, denoted as $\mathcal{G}_i := (\mathcal{N}_i, \mathcal{E}_i)$ for $i=1,\dots,k$. The vertex and arc sets of each copy $\mathcal{G}_i$ are defined as $\mathcal{N}_i := \{v^i : v \in \mathcal{N}\}$ and $\mathcal{E}_i := \{(u^i, v^i) : (u,v) \in \mathcal{E}\}$, where $v^i$ denotes the $i$-th copy of vertex $v$. Thus, the Q-$k$-VDP is equivalent to finding a path in each $\mathcal{G}_i$ connecting the respective copies of $s_i$ and $t_i$ such that the total quadratic cost is minimized and, for any vertex $v \in \mathcal{N}$, at most one of its $k$ copies is visited.

To represent these $k$ paths within a single framework, we consider the disjoint union graph $k\mathcal{G} := \bigsqcup_{i=1}^k \mathcal{G}_i$, which comprises $m = k|\mathcal{N}|$ vertices and $n = k|\mathcal{E}|$ arcs. Let $A \in \mathbb{R}^{m \times n}$ be the vertex--arc incidence matrix of the graph $k\mathcal{G}$, and define the supply--demand vector $b \in \mathbb{R}^{m}$ such that for each $i \in \{1, \dots, k\}$, the entry $b_{v^i}$ corresponding to vertex $v^i \in \mathcal{N}_i$ is $1$ if $v = s_i$, $-1$ if $v = t_i$, and $0$ otherwise.

Let $x \in \{0,1\}^n$ be the binary decision vector where $x_e = 1$ if arc $e$ in $k\mathcal{G}$ is selected, and $0$ otherwise. The Q-$k$-VDP is formulated as the following binary quadratic program (BQP):
\begin{subequations}\label{eq:BQP-form}
\begin{align}
\min_{x} \quad & x^\top \tilde{Q} x \label{BQP-obj}\\
\text{s.t.} \quad 
& Ax = b, \label{BQP-flow}\\
& \sum_{i=1}^k \sum_{e \in \delta_{+}(v^i)} x_e \leq 1, \quad \forall v \in \mathcal{N}, \label{BQP-out}\\
& \sum_{i=1}^k \sum_{e \in \delta_{-}(v^i)} x_e \leq 1, \quad \forall v \in \mathcal{N}, \label{BQP-in}\\
& \sum_{e \in \mathcal{E}_i(S)} x_e \leq |S| - 1, \quad \forall i \in \{1, \dots, k\}, \; \forall S \subseteq \mathcal{N}_i, |S| \geq 2, \label{BQP-sec}\\
& x \in \{0,1\}^n, \label{BQP-bin}
\end{align}
\end{subequations}
where $\tilde{Q} \in \mathbb{S}^n$ is the symmetric indefinite cost matrix capturing both individual arc costs and pairwise interactions. Constraint \eqref{BQP-flow} ensures flow conservation from sources to targets. Constraints \eqref{BQP-out} and \eqref{BQP-in} guarantee vertex-disjointness. Finally, the subtour-elimination constraints (SECs)\eqref{BQP-sec} are imposed to exclude cycles disconnected from the source--target paths.
 
Building upon the BQP, we will later derive an SDP relaxation. To ensure this relaxation is sufficiently tight, we identify a conflict set $\mathcal{C}$ consisting of arc pairs $(p,q)$ that are mutually exclusive:
\begin{equation*}
\mathcal{C} := \left\{ (p,q) : 
\begin{aligned}
    & \forall v \in \mathcal{N}, i \neq j \text{ s.t. } p \in \delta(v^i), q \in \delta(v^j), \quad \text{or} \\
    & \forall v \in \mathcal{N}, \forall i, p \neq q \text{ s.t. } \{p, q\} \subseteq \delta_{+}(v^i) \text{ or } \{p, q\} \subseteq \delta_{-}(v^i)
\end{aligned} \right\}.
\end{equation*}

Here, the first condition ensures vertex-disjointness across different paths, while
the second enforces that, within the disjoint union graph $k\mathcal{G}$, at most one incoming arc and at most one outgoing arc can be active at a vertex.
Although these relations are implicitly satisfied by the linear constraints in \eqref{eq:BQP-form}, explicitly identifying this conflict set is essential for strengthening the subsequent SDP relaxation.

\section{Disjoint Union Graph Reduction}\label{sec:dedu}
 For large $k$, the size of $k\mathcal{G}$ imposes a high-dimensional and computationally demanding decision space on the BQP. This section demonstrates that it is possible to simplify $k\mathcal{G}$. Due to the vertex-disjointness constraints, certain arcs in $k\mathcal{G}$ must appear in every feasible solution of the BQP, while others cannot appear in any feasible solution. By identifying and removing these fixed arcs, we obtain a reduced graph. We then formulate a reduced binary quadratic program (R-BQP) on this reduced graph, which is equivalent to the original BQP while operating in a significantly lower-dimensional space.
 
Since arc fixity depends only on feasibility in the Q-$k$-VDP and not on the quadratic objective, it can be determined by analyzing the corresponding classical $k$-VDP. For fixed $k$, the $k$-VDP is solvable in polynomial time on planar graphs \cite{schrijver1994finding}. Leveraging this result, we introduce a polynomial-time reduction procedure for planar graphs to systematically identify and remove fixed arcs.

\subsection{Fixed-Zero Arcs}

We begin by identifying arcs that cannot appear in any feasible solution to \eqref{eq:BQP-form}. Formally, we define an arc $e=(u^i,v^i)$ in $k\mathcal{G}$ as a fixed-zero arc if $x_e = 0$ for every feasible solution $x$ to \eqref{eq:BQP-form}. To determine if $e$ is a fixed-zero arc, we must check if the underlying arc $(u,v)$ in the base graph $\mathcal{G}$ is ever used by a path from $s_i$ to $t_i$. For this purpose, we construct an auxiliary $(k+1)$-VDP problem instance on $\mathcal{G}$ with the following source--target pairs:
$$(s_1, t_1), \ldots, (s_{i-1}, t_{i-1}), (s_i, u), (v, t_i), (s_{i+1}, t_{i+1}), \ldots, (s_k, t_k).$$

\begin{lemma}
The arc $e = (u^i, v^i)$ in $k\mathcal{G}$ is a fixed-zero arc if and only if the auxiliary $(k+1)$-VDP on $\mathcal{G}$ is infeasible.
\end{lemma}

\begin{proof}
We prove the equivalent statement: there exists a feasible solution $x$ to \eqref{eq:BQP-form} with $x_e = 1$ if and only if the auxiliary $(k+1)$-VDP is feasible.

First, suppose $x$ is feasible for \eqref{eq:BQP-form} and $x_e = 1$. Then some $s_{i}$--$t_{i}$ path uses the arc $(u, v)$. Removing $(u, v)$ splits this path into a subpath from $s_i$ to $u$ and a subpath from $v$ to $t_i$. These two subpaths, along with the other $k-1$ paths, are mutually vertex-disjoint and thus form a feasible solution to the auxiliary $(k+1)$-VDP.

Conversely, suppose the auxiliary $(k+1)$-VDP is feasible, so it has $k+1$ mutually disjoint paths. Since $(u,v)$ is an arc in $\mathcal{G}$, we can concatenate the path from $s_{i}$ to $u$ and the path from $v$ to $t_{i}$ into a single path from $s_{i}$ to $t_{i}$. Together with the remaining $k-1$ paths, this constitutes a feasible set of $k$ vertex-disjoint paths for  all $k$ source--target pairs $(s_{i},t_{i})$ for $i =1,\ldots,k$. This yields a feasible solution $x$ to \eqref{eq:BQP-form} where $x_e = 1$.
\end{proof}

\subsection{Fixed-One Arcs}
Next, we identify arcs that must be included in every feasible solution to \eqref{eq:BQP-form}. Formally, we define an arc $e=(u^i,v^i)$ in $k\mathcal{G}$ as a \emph{fixed-one} arc if $x_e = 1$ for every feasible solution $x$ to \eqref{eq:BQP-form}. We identify these arcs by performing a deletion test on the base graph $\mathcal{G}$, incorporating the previously identified fixed-zero arcs.
Let $\tilde{\mathcal{G}}$ be the graph obtained by removing arc $(u,v)$ from $\mathcal{G}$. We then construct an auxiliary $k$-VDP problem instance on $\tilde{\mathcal{G}}$ with the following source--target pairs:
$$(s_1, t_1), \ldots, (s_{i-1}, t_{i-1}), (s_i,t_i), (s_{i+1}, t_{i+1}), \ldots, (s_k, t_k).$$
 
\begin{lemma}
An arc $e = (u^i, v^i)$ in $k\mathcal{G}$ is a fixed-one arc if and only if:
\begin{enumerate}[label=(\alph*)]
    \item The auxiliary $k$-VDP on $\tilde{\mathcal{G}}$ is infeasible;
    \item The arc $(u^j, v^j)$ is a fixed-zero arc in $k\mathcal{G}$ for all $j \neq i$.
\end{enumerate}
\end{lemma}

\begin{proof}
Suppose $e = (u^i, v^i)$ is a fixed-one arc. Then every feasible solution $x$ of \eqref{eq:BQP-form} satisfies $x_e = 1$, which implies any path from $s_{i}$ to $t_{i}$ must use the arc $(u, v)$. Hence, the auxiliary $k$-VDP on $\tilde{\mathcal{G}}$ is infeasible, satisfying condition (a). Furthermore, by vertex-disjointness, if the $i$-th path uses $(u, v)$, no other path $j \neq i$ can use it. This forces $x_{(u^j, v^j)} = 0$, satisfying condition (b).

Conversely, assume (a) and (b) hold. The infeasibility of the auxiliary $k$-VDP implies that the arc $(u,v)$ must be used by at least one source--target pair to form a feasible solution. Since $(u^j, v^j)$ is a fixed-zero arc for all $j \neq i$, no other path $j \neq i$ can use $(u, v)$. The $i$-th path must then use $(u, v)$. This gives $x_e = 1$, making $e$ a fixed-one arc.
\end{proof}

\begin{corollary}\label{induced-zero}
If $e = (u^i, v^i)$ is a fixed-one arc in $k\mathcal{G}$, the vertex-disjoint constraints \eqref{BQP-out} and \eqref{BQP-in} logically induce a set of fixed-zero arcs in $k\mathcal{G}$.
\end{corollary}

\begin{proof}
Setting $x_e = 1$ saturates the unit capacity of vertices $u$ and $v$ in \eqref{BQP-out} and \eqref{BQP-in}. This saturation forces every arc incident to $u^j$ or $v^j$ to be a fixed-zero arc for all $j \neq i$. It further requires that all outgoing arcs from $u^i$ except $e$, and all incoming arcs to $v^i$ except $e$, must also be fixed-zero arcs.
\end{proof}
\subsection{Graph Reduction}
In this section, we show how identifying fixed arcs simplifies $k\mathcal{G}$ into a reduced graph, yielding an R-BQP equivalent to \eqref{eq:BQP-form}. We denote the reduced graph by $G=(V,E)$.

Let $E^0 \subset \bigsqcup_{i=1}^k \mathcal{E}_i$ and $E^1 \subset \bigsqcup_{i=1}^k \mathcal{E}_i$ denote the sets of arcs fixed to zero and one, respectively. The reduced graph $G$ is constructed by removing the arc set $E^0 \cup E^1$ from $k\mathcal{G}$, yielding a reduced set of arcs $E$ with cardinality $n' = |E|$. Let $A'$ denote the incidence matrix of $G$, and let $b' = b - \sum_{e \in E^1} A_e$ represent the updated supply--demand vector, where $A_e$ denotes the column of $A$ associated with arc $e$.

The R-BQP is then formulated on $G$ using the reduced binary vector $x' \in \{0,1\}^{n'}$. By substituting the fixed values $x_{E^1} = \mathbf{1}$ and $x_{E^0} = \mathbf{0}$ into \eqref{BQP-obj}, we derive a new quadratic objective $(x')^\top Q' x' + (c')^\top x' + \kappa$, where $Q' \in \mathbb{S}^{n'}$, $c' \in \mathbb{R}^{n'}$, and $\kappa \in \mathbb{R}$ is a constant. The flow conservation constraints are replaced by $A' x' = b'$, while the vertex-disjointness constraints and the subtour-elimination constraints are imposed only on the vertices and arcs that remain in $G$. We next show that the R-BQP is equivalent to the BQP.

While fixing arcs to zero trivially eliminates variables, fixing arcs to one alters the flow conservation constraints. Since $k\mathcal{G}$ is a disjoint union of $\mathcal{G}_i$, each fixed-one arc belongs exclusively to a single component $\mathcal{G}_i$. Therefore, we first restrict our analysis of fixed-one arcs to a directed graph $\mathcal{G}=(\mathcal{N},\mathcal{E})$ with a single source--target pair $(s, t)$. The following lemmas establish that these arcs must be used in the same order in every path from $s$ to $t$ and naturally partition the vertex set.

\begin{lemma}\label{orderE}
Let $\mathcal{G} = (\mathcal{N},\mathcal{E})$ be a directed graph, and let $s, t \in \mathcal{N}$ be the source and target vertices, respectively. If a set of arcs $\mathcal{E}^* \subseteq \mathcal{E}$ must be included in every feasible path from $s$ to $t$, then the arcs in $\mathcal{E}^*$ must  appear in the same order in every $s$--$t$ path.
\end{lemma}

\begin{proof}
Suppose for contradiction that two arcs $e_1, e_2 \in \mathcal{E}^*$ can appear in different orders. Let $P_1$ be a feasible $s$--$t$ path visiting $e_1$ before $e_2$, and let $P_2$ be a feasible $s$--$t$ path visiting $e_2$ before $e_1$. Let $u$ be the tail vertex of $e_2$.

We can form a new walk $P$ by combining the segment of $P_2$ from $s$ up to $u$ with the segment of $P_1$ from $u$ to $t$. We first show that $P$ is a simple path. Suppose there exists a vertex $w \neq u$ shared by the segment $P_2[s,u]$ and the segment $P_1[u,t]$. If such a vertex $w$ existed, we could construct an $s$--$t$ path passing through $w$. This path would bypass $e_1$ and $e_2$ entirely, contradicting the premise that $e_1, e_2 \in \mathcal{E}^*$. Thus, $P$ must be a simple path.

Then, because $e_1$ is located after $e_2$ in $P_2$ and before $e_2$ in $P_1$, this concatenated path $P$ entirely avoids $e_1$. This contradicts the premise that $e_1$ must be used in every feasible path from $s$ to $t$. Therefore, all arcs in $\mathcal{E}^*$ must follow the same order in every feasible $s$--$t$ path.
\end{proof}

\begin{lemma} \label{orderV}
Let $\mathcal{E}^* = \{e_1, \dots, e_g\}$ be the set of arcs that appear in every $s$--$t$ path in $\mathcal{G}$, indexed by the order in which they appear along every such path, where $e_i=(a_i,b_i)$. These arcs partition the set of all vertices that lie on at least one $s$--$t$ path into $g+1$ pairwise disjoint subsets $V_0, V_1, \dots, V_g$. Furthermore, in the reduced graph $\mathcal{G}' = (\mathcal{N}, \mathcal{E} \setminus \mathcal{E}^*)$, there is no arc $(u, v)$ such that $u \in V_i$ and $v \in V_j$ for any $i < j$.
\end{lemma}

\begin{proof}
The arcs in $\mathcal{E}^*$ decompose any $s$--$t$ path into exactly $g+1$ independent subpaths: from $s$ to $a_1$, from $b_i$ to $a_{i+1}$ for $1 \le i \le g-1$, and from $b_g$ to $t$. By definition, none of these subpaths contain any arcs from $\mathcal{E}^*$.

Let $V_0$ contain all vertices $v$ such that there is a path in $\mathcal{G}'$ from $s$ to $a_1$ containing $v$. For $1 \le i \le g-1$, let $V_i$ contain all vertices $v$ such that there is a path in $\mathcal{G}'$ from $b_i$ to $a_{i+1}$ containing $v$. Finally, let $V_g$ contain all vertices $v$ such that there is a path in $\mathcal{G}'$ from $b_g$ to $t$ containing $v$. For each $V_i$, let $o_i$ and $d_i$ be its supply and demand vertices: $(o_0, d_0) = (s, a_1)$, $(o_g, d_g) = (b_g, t)$, and $(o_i, d_i) = (b_i, a_{i+1})$ for $1 \le i < g$.

Suppose $v \in V_i \cap V_j$ for $i < j$. By definition, paths exist in $\mathcal{G}'$ from $o_i$ to $v$ and from $v$ to $d_j$. Concatenating a valid $s$--$o_i$ path and a valid $d_j$--$t$ path creates an $s$--$t$ path in $\mathcal{G}$. Because $i < j$, jumping directly from $o_i$ to $d_j$ via $v$ completely bypasses the arcs $e_{i+1}, \dots, e_j$. This contradicts the premise that $e_j$ must appear in every $s$--$t$ path. Thus, these subsets are strictly disjoint.

Suppose an arc $(u, v) \in \mathcal{G}'$ exists from $V_i$ to $V_j$ for $i < j$. Paths exist in $\mathcal{G}'$ from $o_i$ to $u$ and from $v$ to $d_j$. Combined with $(u, v)$, they form a path from $o_i$ to $d_j$ that uses no arcs from $\mathcal{E}^*$. Prepending an $s$--$o_i$ path and appending a $d_j$--$t$ path again forms an $s$--$t$ path in $\mathcal{G}$ that bypasses $e_j$, resulting in a contradiction. Therefore, no such arc exists in $\mathcal{G}'$.
\end{proof}
\begin{theorem}\label{singlepair}
Let $A'$ be the incidence matrix of $\mathcal{G}'$, and $b'$ be the updated supply--demand vector with supplies of 1  at $\{s, b_1, \dots, b_g\}$ and demands of -1 at $\{a_1, \dots, a_g, t\}$. Any acyclic binary solution $x'$ satisfying $A'x' = b'$ forms, together with $\mathcal{E}^*$, a valid $s$--$t$ path in $\mathcal{G}$.
\end{theorem}
\begin{proof}
According to Lemma \ref{orderV}, in the reduced graph $\mathcal{G}'$, no arc exists from $V_i$ to $V_j$ for $i < j$. Hence, the supply originating in any subset $V_i$ cannot reach the demand in any subset $V_j$ where $i < j$. Therefore, to ensure that all flow requirements are met, the supply within each subset $V_i$ must be strictly satisfied by the demand within that same $V_i$. Specifically, the supply $+1$ at $s$ must reach $a_1$ in $V_0$, each supply $+1$ at $b_k$ must reach $a_{k+1}$ in $V_k$ for $1 \le k < g$, and the supply $+1$ at $b_g$ must reach $t$ in $V_g$.

Consequently, the acyclic binary solution $x'$ necessarily forms $g+1$ subpaths $P_0, P_1, \dots, P_g$, where each $P_i$ is strictly contained within $V_i$, connecting $o_i$ to $d_i$. Since each $e_i = (a_i, b_i)$ links the destination $a_i$ of subpath $P_{i-1}$ to the origin $b_i$ of subpath $P_i$, the union $P_0 \cup {e_1} \cup P_1 \cup \dots \cup {e_g} \cup P_g$ constitutes an $s$--$t$ path in the original graph $\mathcal{G}$. Thus, the non-zero entries of $x'$ combined with the set $\mathcal{E}^*$ yield a valid $s$--$t$ path.
\end{proof}

Since $k\mathcal{G} = \bigsqcup_{i=1}^k \mathcal{G}_i$ is a disjoint union, any vector $x$ feasible to \eqref{eq:BQP-form} and the fixed-one arc set $E^1$ naturally partition into $k$ components $x^i$ and $E^1_i$ for each $\mathcal{G}_i$.  Similarly, the parameters $A$ and $b$ can be block-partitioned as $A_i$ and $b_i$, while the parameters $A'$ and $b'$ can be block-partitioned as $A'_i$ and $b'_i$.  
\begin{theorem} \label{theo}
Let $x$ be a binary vector on $k\mathcal{G}$ defined by $x_E = x'$, $x_{E^1} = \mathbf{1}$, and $x_{E^0} = \mathbf{0}$. Then, $x'$ is feasible for the R-BQP if and only if $x$ is feasible for the BQP. Moreover, their respective objective values coincide.
\end{theorem}

\begin{proof}
 
Sufficiency ($\Rightarrow$): 
Suppose $x'$ is feasible for the R-BQP. The disjoint structure of the graph implies that $x'$ decomposes into $k$ sub-vectors $(x')^i$, each satisfying $A_i' (x')^i = b_i'$. By Theorem \ref{singlepair}, the union of the non-zero entries in $(x')^i$ and $E^1_i$ forms a valid $s_i$--$t_i$ path within each $\mathcal{G}_i$.

Let $\mathcal{N}_{E^1}\subset \mathcal{N}$ denote the base vertices occupied by the fixed-one arcs $E^1$. Then, Corollary \ref{induced-zero} establishes that all other arcs associated with $\mathcal{N}_{E^1}$ in $k\mathcal{G}$ are fixed to zero. Since these induced fixed-zero arcs are entirely contained within the removed set $E^0$, the reduced graph $G$ structurally enforces that $x'$ cannot visit any copies of $\mathcal{N}_{E^1}$.  Furthermore, the vector $x'$ satisfies the vertex-disjoint constraints on the remaining free vertices. Therefore, the vector $x$ forms $k$ valid, vertex-disjoint paths, proving its feasibility for the BQP on $k\mathcal{G}$.

Necessity ($\Leftarrow$): 
Suppose $x$ is feasible for the BQP. Since $G$ is a subgraph of $k\mathcal{G}$, let $x'$ be the restriction of $x$ to the arc set of $G$. As $x$ satisfies the global flow conservation $Ax = b$ and $E^1$ accounts for the flow at its endpoints, $x'$ necessarily satisfies the updated flow balance $A'x' = b'$. Moreover, since $x$ satisfies the vertex-disjointness constraints on $k\mathcal{G}$, its restriction $x'$ must also satisfy these constraints on $G$. Therefore, $x'$ is feasible for the R-BQP.

The objective function of the BQP is given by $f(x) = x^\top \tilde{Q} x$. By partitioning the arcs into $E, E^1$, and $E^0$, and substituting $x_E = x'$, $x_{E^1} = \mathbf{1}$, and $x_{E^0} = \mathbf{0}$, we have:
\begin{equation*}
f(x) = (x')^\top \tilde{Q}_{EE} x' + \sum_{e \in E} \sum_{f \in E^1} (\tilde{Q}_{ef} + \tilde{Q}_{fe}) x'_e + \sum_{e \in E^1} \sum_{f \in E^1} \tilde{Q}_{ef}.
\end{equation*}
By setting $Q' = \tilde{Q}_{EE}$, $c'_e = \sum_{f \in E^1} (\tilde{Q}_{ef} + \tilde{Q}_{fe})$, and $\kappa = \sum_{e \in E^1} \sum_{f \in E^1} \tilde{Q}_{ef}$, the resulting expression coincides with the objective function $f_R(x')$ of the R-BQP. Thus, $f(x) = f_R(x')$.
\end{proof}

\section{SDP Relaxation of the Subtour-Relaxed Model}\label{sec:sdp_relaxation}

The BQP is computationally intractable due to the combination of its nonconvex quadratic objective and the exponential number of SECs. To establish a tractable bounding framework, we consider a subtour-relaxed BQP by omitting the SECs from the BQP. While the omission of SECs may allow solutions with cycles, this relaxation still provides a valid lower bound for the original BQP. However, even this subtour-relaxed BQP remains NP-hard and difficult to solve directly as a nonconvex quadratic optimization problem over binary variables. Therefore, we develop SDP relaxations for the subtour-relaxed BQP to provide high-quality lower bounds and yield more robust performance for large-scale instances.

\subsection{SDP Relaxation and Facial Reduction}

To derive an SDP relaxation of the subtour-relaxed BQP, we define a matrix variable $Y \in \mathbb{S}_{+}^{n+1}$, indexed from $0$ to $n$. We impose linear constraints on $Y$ to ensure that the feasible region contains all rank-one matrices of the form:
$$
\begin{bmatrix} 1 \\ x \end{bmatrix} \begin{bmatrix} 1 & x^\top \end{bmatrix} = \begin{bmatrix} 1 & x^\top \\ x & xx^\top \end{bmatrix},
$$
where $x \in \{0,1\}^n$ is a feasible solution to the subtour-relaxed BQP. To capture the binary nature of $x$, we impose the arrow constraint:
$$
\text{arrow}(Y) = \begin{bmatrix}
    Y_{0,0}\\
    Y_{1,1} - Y_{0,1} \\
    \vdots\\
    Y_{n,n} - Y_{0,n}
\end{bmatrix} = e_0,
$$
where $e_{0}$ is the first unit vector in $\mathbb{R}^{n+1}$. Since $Y \in \mathbb{S}_{+}^{n+1}$, the condition $Y_{i,i} = Y_{0,i}$ models the requirement $x_e^2 = x_e$. Next, we enforce the flow conservation constraints $Ax = b$. By defining $M_A = [-b \mid A]^\top [-b \mid A]$, this requirement is linearized as the trace constraint $\langle M_A, Y \rangle = 0$. In addition, the quadratic objective $x^\top \tilde{Q} x$ is linearized as $\langle Q, Y \rangle$ by defining the block matrix
$$
Q = \begin{bmatrix} 0 & 0 \\ 0 & \tilde{Q} \end{bmatrix}.
$$
Combining these elements, we obtain the following SDP relaxation:
\begin{equation} \label{SDP-form}
\begin{aligned}
\min_{Y} \quad & \langle Q, Y \rangle \\
\text{s.t.} \quad & \langle M_A, Y \rangle = 0, \\
& Y_{p,q} = 0, \quad \forall (p,q) \in \mathcal{C}, \\
& \operatorname{arrow}(Y) = e_0, \\
& Y \succeq 0,
\end{aligned}
\end{equation}
where $\mathcal{C}$ is the conflict set defined in Section \ref{sec:bqp}.

To improve computational efficiency, we apply an analytic facial reduction step to \eqref{SDP-form}. The flow conservation constraints in \eqref{SDP-form} imply that the feasible region of \eqref{SDP-form} lies within the face $\mathcal{F} = \{Y \succeq 0 \mid \langle M_A, Y \rangle = 0\}.$ Since $M_A \succeq 0$, any $Y \in \mathcal{F}$ must satisfy $\text{range}(Y) \subseteq \ker(M_A)$. Let $r = \dim(\ker(M_A))$ and let $V \in \mathbb{R}^{(n+1) \times r}$ be a matrix whose columns form an orthonormal basis for $\ker(M_A)$. Any feasible $Y$ can therefore be parameterized as $Y = V R V^\top$ for $R \in \mathbb{S}_{+}^r$. Substituting this into \eqref{SDP-form} yields the facially reduced SDP:
\begin{equation} \label{sdpfr}
\begin{aligned}
\min_{Y, R} \quad & \langle Q, Y \rangle \\
\text{s.t.} \quad & Y = V R V^\top, \\
& Y_{p,q} = 0, \quad \forall (p,q) \in \mathcal{C}, \\
& \operatorname{arrow}(Y) = e_0, \\
& R \succeq 0,
\end{aligned}
\end{equation}

By solving for $R$ in the reduced space $\mathbb{S}_{+}^r$, we significantly decrease the number of variables and accelerate convergence, particularly in large-scale instances where $r \ll n+1$.

\subsection{Slater's Condition}
Even though $\mathbb{S}_{+}^r$ restricts the problem to a lower-dimensional face, the resulting system \eqref{sdpfr} may not satisfy Slater's condition. To certify the failure of Slater's condition, we consider the facial reduction auxiliary system associated with \eqref{sdpfr}:
\begin{equation}\label{aux_slater}
\begin{aligned}
\min_{X,z,y,W}\quad & \langle z,e_0\rangle \\
\text{s.t.}\quad
& X-\textnormal{arrow}^*(z)+\psi^*(y)=0  \\
& V^\top XV = W \\
& X\in\mathbb{S}^{n+1},\, z\in\mathbb{R}^{n+1},\, y\in\mathbb{R}^{|\mathcal{C}|},\, W\in\mathbb{S}_{+}^{r} 
\end{aligned}  
\end{equation}

The adjoint $\textnormal{arrow}^*:\mathbb{R}^{n+1}\to\mathbb{S}^{n+1}$ is defined by $\langle \textnormal{arrow}^*(z),Y\rangle=\langle z,\textnormal{arrow}(Y)\rangle$. Let $L_{ij}$ be the matrix with a $1$ at entry $(i,j)$ and zeros elsewhere. Then we have
\[
\textnormal{arrow}^*(z) = z_0 L_{00} + \sum_{i=1}^{n} z_i\Big(L_{ii}-\tfrac12 L_{0i}-\tfrac12 L_{i0}\Big).
\]
Define $\psi:\mathbb{S}^{n+1}\to\mathbb{R}^{|\mathcal{C}|}$ such that $(\psi(Y))_{(p,q)} = Y_{pq}$ for $(p,q)\in\mathcal{C}$, and its adjoint $\psi^*:\mathbb{R}^{|\mathcal{C}|}\to\mathbb{S}^{n+1}$ is defined by $(\psi^*(y))_{pq} = y_{pq}$ for $(p,q)\in\mathcal{C}$, and zero otherwise.
 
If the optimal value of \eqref{aux_slater} is non-positive (i.e., $\langle z, e_0 \rangle \le 0$) and there exists a non-zero solution $W \succeq 0$, then $W$ acts as the exposing vector certifying that Slater's condition fails for \eqref{sdpfr}. Conversely, if $W=0$ is the unique solution for all $\langle z, e_0 \rangle \le 0$, the problem is strictly feasible. While this system follows from the Theorem of the alternative, a concise derivation for this auxiliary system is provided in Appendix~\ref{app:aux_derivation} for completeness.

\begin{example} \label{ex:slater_failure}
Consider the directed planar graph $\mathcal{G}$ shown in Figure~\ref{fig:example_5_1}. We seek two vertex-disjoint paths: $1 \to 2$ and $3 \to 4$. To model this, we construct the disjoint union $\mathcal{G}^{1} \sqcup \mathcal{G}^{2}$ as shown in Figure~\ref{fig:disjoint_union}, where each component $\mathcal{G}^{i}$ is isomorphic to the base graph $\mathcal{G}$. 

\begin{figure}[ht]
    \centering
    \begin{tikzpicture}[>=Stealth, scale=0.7, every node/.style={circle,draw,minimum size=6mm, font=\scriptsize}]   
      \node (1) at (0,0.5) {1}; 
      \node (2) at (3,2.2) {2}; 
      \node (3) at (1.5,1.15) {3};
      \node (4) at (4.0,3.2) {4}; 
      \node (5) at (1.5,2.8) {5}; 
      \node (6) at (4.0,0.5) {6};
      \draw[->] (1)--(5); \draw[->] (1)--(6); \draw[->] (3)--(5); \draw[->] (3)--(6);
      \draw[->] (5)--(2); \draw[->] (6)--(2); \draw[->] (5)--(4); \draw[->] (6)--(4);
    \end{tikzpicture}
    \caption{Directed planar graph $\mathcal{G}$ with source--target pairs $(1, 2)$ and $(3, 4)$ as described in Example~\ref{ex:slater_failure}. This instance, consisting of six nodes and eight arcs, is used to illustrate the failure of Slater's condition in the SDP relaxation.}
    \label{fig:example_5_1}
\end{figure}
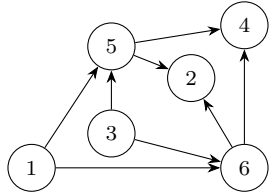

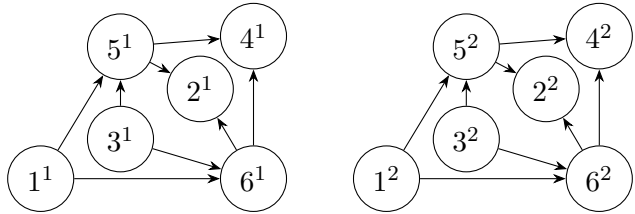
\begin{figure}[htbp]
    \centering
    \begin{tikzpicture}[>=Stealth, scale=0.7, every node/.style={circle,draw,minimum size=6mm, font=\small}]
      \node (1a) at (0,0.5) {$1^{1}$}; 
      \node (2a) at (3,2.2) {$2^{1}$}; 
      \node (3a) at (1.5,1.25) {$3^{1}$};
      \node (4a) at (4.0,3.2) {$4^{1}$}; 
      \node (5a) at (1.5,3.0) {$5^{1}$}; 
      \node (6a) at (4.0,0.5) {$6^{1}$};
    
      \draw[->] (1a)--(5a); \draw[->] (1a)--(6a); \draw[->] (3a)--(5a); \draw[->] (3a)--(6a);
      \draw[->] (5a)--(2a); \draw[->] (6a)--(2a); \draw[->] (5a)--(4a); \draw[->] (6a)--(4a);
    
      \begin{scope}[xshift=6.5cm]
        \node (1b) at (0,0.5) {$1^{2}$}; 
        \node (2b) at (3,2.2) {$2^{2}$}; 
        \node (3b) at (1.5,1.25) {$3^{2}$};
        \node (4b) at (4.0,3.2) {$4^{2}$}; 
        \node (5b) at (1.5,3.0) {$5^{2}$}; 
        \node (6b) at (4.0,0.5) {$6^{2}$};
    
        \draw[->] (1b)--(5b); \draw[->] (1b)--(6b); \draw[->] (3b)--(5b); \draw[->] (3b)--(6b);
        \draw[->] (5b)--(2b); \draw[->] (6b)--(2b); \draw[->] (5b)--(4b); \draw[->] (6b)--(4b);
      \end{scope}
    \end{tikzpicture}
    \caption{Disjoint union graph $\mathcal{G}^{1}\sqcup\mathcal{G}^{2}$, constructed from two isomorphic copies of the base graph $\mathcal{G}$ in Figure~\ref{fig:example_5_1}.} \label{fig:disjoint_union}
\end{figure}

In this instance, the original variable $Y$ resides in $\mathbb{S}_{+}^{17}$. After analytical facial reduction, the problem is restricted to matrices parameterized by $R \in \mathbb{S}_+^7$. To certify the failure of Slater's condition, we solve the auxiliary system \eqref{aux_slater}. The system yields an optimal value of $0$, identifying a non-zero certificate $W \in \mathbb{S}_{+}^{7}$. The existence of this exposing vector shows that the feasible region of \eqref{sdpfr} is contained within a proper face of $\mathbb{S}_{+}^{7}$. Consequently, there is no strictly feasible point $R \succ 0$ in the current space. (The detailed numerical entries of $W$ are provided in Appendix~\ref{app:W_matrix} for completeness.)
\end{example}

\subsection{Dimensionality Reduction and Bound Analysis}

Recall that when the base graph $\mathcal{G}$ is planar, the reduction  procedures detailed in Section \ref{sec:dedu} can be performed in polynomial time. We now consider the subtour-relaxed version of the R-BQP obtained by dropping the  SECs from R-BQP. To avoid ambiguity, we distinguish between model reduction (BQP versus R-BQP) and subtour relaxation (with or without SECs), which gives rise to four formulations: the BQP, the R-BQP, the subtour-relaxed BQP, and the subtour-relaxed R-BQP.

In this subsection, we show that the reduction procedure substantially strengthens the relaxation beyond merely reducing the problem size. Although the BQP and the R-BQP are equivalent, their subtour-relaxed counterparts are not necessarily equivalent. This distinction is formalized in the following observation.

\begin{proposition}
An arc $e$ fixed to zero in the BQP may take a non-zero value in the subtour-relaxed BQP if it participates in a negative cycle that satisfies flow conservation and vertex-disjointness constraints.
\end{proposition}

Consequently, in the graph $k\mathcal{G}$, the subtour-relaxed BQP permits negative cycles containing arcs that are fixed to zero in the BQP,  which lowers the objective value. Since such arcs are absent in the reduced graph $G$, the subtour-relaxed R-BQP is restricted to a strictly tighter feasible region. This leads to the following relationship regarding the relaxation quality:

\begin{proposition} 
Let $v_{f}$ and $v_{r}$ be the optimal values of the subtour-relaxed BQP and the subtour-relaxed R-BQP, respectively. Let $v^*$ denote the optimal value of the original BQP (which, by Theorem ~\ref{theo}, is also the optimal value of the R-BQP). The following relationship holds: 
\begin{equation*} 
v_{f} \le v_{r} \le v^*. 
\end{equation*} 
\end{proposition}

We now consider the SDP relaxation of the subtour-relaxed R-BQP, which is constructed following the same lifting procedure as in Section \ref{sec:sdp_relaxation}. Here, the augmented objective matrix is formed by lifting the reduced parameters $Q'$, $c'$, and $\kappa$ into the space $\mathbb{S}^{n'+1}$. By constructing the lifted matrix $M_{A'}$ from the incidence matrix $A'$ and supply--demand  vector $b'$ of the reduced graph $G$, 
we obtain its facially reduced SDP relaxation by applying the same facial reduction procedure. Let $r' = \dim(\ker(M_{A'}))$. As illustrated in Example ~\ref{ex:slater_failure_red}, we can obtain $r' < r$. Thus, the resulting facially reduced SDP operates in a lower-dimensional space $\mathbb{S}^{r'}_+$.

\subsection{Slater's condition on the Reduced Graph}
 
However, Slater's condition may also fail for the resulting facially reduced SDP of the subtour-relaxed R-BQP.  
\begin{example}  
\label{ex:slater_failure_red}
We apply the reduction procedure to the disjoint union $\mathcal{G}^{1} \sqcup \mathcal{G}^{2}$ in Figure~\ref{fig:disjoint_union} to eliminate all fixed arcs. For the path in $\mathcal{G}^1$, the following arcs are identified as fixed-zero and removed:
\begin{equation*}
    \{(3^1,5^1), (3^1,6^1), (5^1,4^1), (6^1,4^1)\};
\end{equation*}
Similarly, for the path in $\mathcal{G}^2$, we remove these fixed-zero arcs:
\begin{equation*}
    \{(1^2,5^2), (1^2,6^2), (5^2,2^2), (6^2,2^2)\}.
\end{equation*}
There is no fixed-one arc, so the resulting reduced graph $G$ is shown in Figure~\ref{fig:reduced_graph}. 

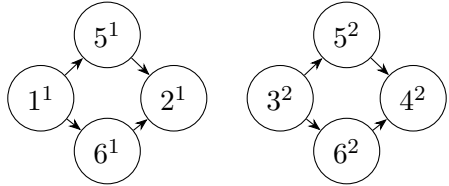
\begin{figure}[htbp]
    \centering
    \begin{tikzpicture}[>=Stealth, scale=0.7, vertex/.style={circle,draw,minimum size=8mm, font=\small}]
      \node[vertex] (11) at (0,0.6) {$1^1$};
      \node[vertex] (21) at (2.5,0.6) {$2^1$};
      \node[vertex] (51) at (1.25,1.8) {$5^1$};
      \node[vertex] (61) at (1.25,-0.4) {$6^1$};
      \draw[->] (11)--(51); \draw[->] (51)--(21);
      \draw[->] (11)--(61); \draw[->] (61)--(21);
    
      \node[vertex] (32) at (4.5,0.6) {$3^2$};
      \node[vertex] (42) at (7.0,0.6) {$4^2$};
      \node[vertex] (52) at (5.75,1.8) {$5^2$};
      \node[vertex] (62) at (5.75,-0.4) {$6^2$};
      \draw[->] (32)--(52); \draw[->] (52)--(42);
      \draw[->] (32)--(62); \draw[->] (62)--(42);
    \end{tikzpicture}
    \caption{Reduced graph $G$ after removing fixed arcs in the disjoint union graph in Figure~\ref{fig:disjoint_union}.}
    \label{fig:reduced_graph}
\end{figure}

Consequently, the facially reduced SDP relaxation for the subtour-relaxed R-BQP on $G$ is restricted to a face isomorphic to $\mathbb{S}_+^3$. Solving the corresponding facial reduction auxiliary system yields an optimal value of $0$ and the following non-zero solution $W \in \mathbb{S}_{+}^{3}$:
\begin{equation*}
W = \begin{bmatrix}
0.5 & -0.5 & 0 \\
-0.5 & 0.5 & 0 \\
0 & 0 & 0
\end{bmatrix} \succeq 0.
\end{equation*}
The existence of this $W \neq 0$ provides a numerical certificate that the feasible region of the facially reduced SDP remains on the boundary of the positive semidefinite cone.
\end{example}

\section{The Alternating Direction Method of Multipliers}\label{sec:admm}
The failure of Slater's condition, together with the high memory requirements, makes interior-point methods numerically unstable and slow for large-scale SDPs. We therefore use ADMM to solve the facially reduced SDP relaxation. As a first-order method, ADMM is memory-efficient because it decomposes the relaxation into simpler, inexpensive subproblems.

To improve numerical stability, we additionally impose the bounds $0 \le Y \le 1$. The problem \eqref{sdpfr} is reformulated as:

\begin{equation} \label{sdpfr_box}
\begin{aligned}
\min_{Y, R} \quad & \langle Q, Y \rangle \\
\text{s.t.} \quad & Y = V R V^\top, \\
& Y_{p,q} = 0, \quad \forall (p,q) \in \mathcal{C}, \\
& \operatorname{arrow}(Y) = e_0, \\
& 0 \le Y \le 1, \\
& R \succeq 0,
\end{aligned}
\end{equation}
 
With the nonnegativity constraint $Y \geq 0$, the upper bound $Y \leq 1$ is redundant, but we retain it because it improves convergence in practice \cite{oliveira2018admm}.

\subsection{ADMM Algorithm and Subproblem Updates}
The augmented Lagrangian function of \eqref{sdpfr_box} corresponding to the linear constraint $Y =  VRV^\top$ is given by:
\begin{equation}\label{Lag-admm}
L(R,Y, Z) = \langle  Q, Y \rangle + \langle Z, Y-VRV^\top \rangle + \frac{\beta}{2} \| Y - VRV^\top \|^2_F,
\end{equation}
 where $R$ and $Y$ are primal variables, and $Z$ is the dual variable. $ \beta >0 $ is the penalty parameter, and $\|\cdot\|_F$ is the Frobenius norm. At each iteration $ \ell $, given the current $ (R^{\ell },Y^{\ell },Z^{\ell }) $, ADMM updates the variables by solving the following subproblems sequentially:
 \begin{subequations}\label{admm-updates}
\begin{align}
R^{\ell +1} &= \arg\min_{R \succeq 0} L(R, Y^{\ell }, Z^{\ell }), \label{eq:admm-R} \\[2pt]
Y^{\ell +1} &= \arg\min_{Y \in \mathcal{F}_Y} L(R^{\ell +1}, Y, Z^{\ell }), \label{eq:admm-Y} \\[2pt]
Z^{\ell +1} &= Z^{\ell } + \rho\, \cdot \beta\,\bigl(Y^{\ell +1} - V R^{\ell +1} V^{\top}\bigr). \label{eq:admm-Z}
\end{align}
\end{subequations}
where 
\begin{equation}\label{admm-feas}
   \mathcal{F}_Y=\Biggl\{\, Y\in\mathbb{S}^{n+1}\ \Bigg|\ 
\begin{aligned}
& Y_{p,q}=0,\quad \forall (p,q)\in \mathcal{C} \\  
&\operatorname{arrow}(Y)=e_0,\\
&0\le Y\le 1
\end{aligned}
\Biggr\}. 
\end{equation}
Here $\rho \in (0,\frac{1+\sqrt{5}}{2})$ is the step-size for updating the dual variable $Z$, see \cite{wen2010alternating}. 

From the augmented Lagrangian, the $R$-subproblem reduces to minimizing a squared Frobenius norm:
\[
\min_{R \succeq 0} \;\Bigl\|\, Y^{(\ell)} - V R V^\top + \tfrac{Z^{(\ell)}}{\beta}\,\Bigr\|_F^2.
\]
Given $V^\top V = I$, the solution is obtained by projecting onto the cone of positive semidefinite matrices
\[
R^{(\ell+1)} = \mathcal{P}_{\succeq 0}\left( V^\top \left( Y^{(\ell)} + \frac{Z^{(\ell)}}{\beta} \right) V \right).
\]
The $Y$-subproblem reduces to a projection onto the set $\mathcal{F}_Y$:
\[ Y^{(\ell+1)} = \arg\min_{Y \in \mathcal{F}_Y} \left\| Y - \left( V R^{(\ell + 1)} V^{\top} - \frac{Q + Z^{(\ell)}}{\beta} \right) \right\|_F^2. \]
Let $\hat{Y} = V R^{(\ell + 1)} V^{\top} - \frac{Q + Z^{(\ell)}}{\beta}$. The closed-form solution is:
\begin{equation*}
Y_{p,q}^{\ell +1} = 
\begin{cases} 
1, & \text{if } p=q=0, \\[6pt]
\min\left\{ 1, \max\left\{ 0, \frac{\hat{Y}_{p,p} + 2\hat{Y}_{0,p}}{3} \right\} \right\}, & \text{if } 0 < p=q \leq n \text{ or } 0 = p < q \leq n, \\[6pt]
\min\left\{ 1, \max\left\{ 0, \hat{Y}_{p,q} \right\} \right\}, & \text{if } 0 < p < q \leq n, \\[6pt]
0, & \forall (p,q) \in \mathcal{C}.
\end{cases}
\end{equation*}
The remaining entries are determined by symmetry, i.e., $Y_{q,p}^{\ell+1} = Y_{p,q}^{\ell+1}$.

\subsection{Bounds Computation}

The output of the ADMM algorithm is used to compute lower and upper bounds for \eqref{sdpfr_box}. The dual feasible set for \eqref{sdpfr_box} is $\mathcal{F}_Z := \{ Z \mid -V^\top Z V \succeq 0 \}$. For any $Z \in \mathcal{F}_Z$, a valid lower bound is given by:
\begin{equation*}
g(Z) = \min_{Y \in \mathcal{F}_Y} \langle Q + Z, Y \rangle.
\end{equation*}
Since the ADMM output $\tilde{Z}$ may not be dual-feasible, we project it as $Z^{\text{proj}} = \mathcal{P}_{\mathcal{F}_Z}(\tilde{Z})$ and compute a valid lower bound as $g(Z^{\text{proj}})$. 

To derive a valid upper bound, we use the ADMM output $\tilde{Y}$ to extract a vector $w \in \mathbb{R}^n$, where each entry $w_p = (\tilde{Y})_{pp}$ for $p=1, \dots, n$. This vector guides the solution of the following linear program:
\begin{equation*}
x^* = \arg\max \{ w^\top x : x \in \mathcal{X} \},
\end{equation*}
where $\mathcal{X}$ denotes the feasible set of the subtour-relaxed BQP. A valid upper bound for \eqref{sdpfr_box} is then computed as $(x^*)^\top \tilde{Q} x^*$.

To solve the subtour-relaxed BQP to optimality, we implement a B\&B framework using the SDP relaxation \eqref{sdpfr_box} as the node relaxation, with ADMM solving the node subproblems. At each node of the branching tree, ADMM computes upper and lower bounds for the current node and updates the best upper bound found so far. 

\section{Numerical Experiments}\label{sec:experiments}

For general directed graphs, we develop an SDP relaxation and an ADMM-based B\&B framework, referred to as SABB, for the subtour-relaxed BQP. For planar instances, after graph reduction, the same SABB framework is applied to the subtour-relaxed R-BQP obtained from the reduced formulation.

In our computational study, all benchmark instances are planar. Therefore, all reported computational results are for the subtour-relaxed R-BQP, solved using the same SABB framework applied to the reduced model data. We focus on planar instances because they are motivated by the physical topology of VLSI layouts, a key application of the Q-$k$-VDP, and because planarity allows us to leverage the polynomial-time reduction procedure developed in this paper. 

We first describe the benchmark instance generation procedure and evaluate the effectiveness of our reduction procedure. We then compare the performance of SABB against Gurobi on these instances.

All experiments were implemented in MATLAB R2025b on a machine with an AMD Ryzen 9 9950X 16-core processor (4.30 GHz) and 96 GB of RAM. We also used Gurobi 11.0.3 under its default settings, except for the time limit and output flag.

\subsection{Instance Generation and Reduction}

Given a fixed number of vertices $m_v$, we first generate an initial grid graph by randomly choosing the number of rows $m_{rows}$ and columns $m_{cols}$ such that $m_{rows} \times m_{cols} = m_v$. Each undirected edge is replaced by two anti-parallel arcs to form a directed graph. For each graph, we randomly select $k$ source--target pairs within the grid. 
 
To focus on non-trivial instances, we then apply a preprocessing procedure to the grid to obtain our actual problem instance, defined by the graph $\mathcal{G} = (\mathcal{N}, \mathcal{E})$ and the corresponding set of source--target pairs. While $\mathcal{G}$ remains planar, it is no longer a regular grid, as the preprocessing strictly enforces the following structural conditions:

\begin{enumerate}[label=(\roman*)]
    \item Sources have no incoming arcs, and targets have no outgoing arcs.
    \item Every source must have at least two outgoing arcs, and every target must have at least two incoming arcs. If a terminal has only one incident arc, the next step is forced. We therefore move the source forward (or the target backward) along this arc until a branching choice exists.
\end{enumerate}
Following this procedure, we generate 40 feasible instances for each of the 29 distinct configurations, where each configuration is characterized by a unique pair $(m_v, k)$.

Next, we construct the disjoint union graph $k\mathcal{G}$ for each instance, which contains a total of $k|\mathcal{E}|$ arcs. To identify and remove the arcs fixed to zero or one from $k\mathcal{G}$, we use Gurobi to solve the two auxiliary problems defined in Section \ref{sec:dedu}. Let $G = (V, E)$ be the resulting reduced graph, which comprises the remaining unfixed arcs and their incident vertices.

Table \ref{tab:combined_stats} summarizes the performance statistics for each configuration, including the average arc counts for both the union graph $k\mathcal{G}$ and the reduced graph $G$, as well as the average computational time required for this reduction process. These statistics reveal three clear trends:

\begin{enumerate}[label=(\roman*)]
    
    \item \textbf{Impact of $k$:} For a given initial size $m_v$, a larger $k$ naturally increases the likelihood of path conflicts. Because of this higher terminal density, we are able to eliminate a significantly larger proportion of arcs from $k\mathcal{G}$. Notably,
    for the instances with the largest $k$, we successfully eliminate over 95\% of the initial arcs.

     \item \textbf{Impact of $m_v$:} Conversely, for a fixed $k$, the reduction rate decreases as the initial grid size $m_v$ increases. This behavior is expected because instances derived from larger base grids provide greater routing flexibility. With terminals more sparsely distributed across the topology, paths are less likely to conflict, resulting in fewer fixed arcs.
 
    \item \textbf{Computational Efficiency:} Although the reduction time scales with the overall instance dimensions ($m_v$ and $k$), the reduction procedure remains highly efficient in practice. Even the most time-consuming instances are fully reduced in under one minute.
\end{enumerate}

To further evaluate the effectiveness of our reduction procedure, we compare Gurobi's computational performance on the subtour-relaxed BQP against its reduced counterpart, the subtour-relaxed R-BQP, using instances from the $(m_v, k) = (80, 10)$ and $(100, 13)$ configurations. Under a 600-second time limit per instance, Gurobi solved only $3$ and $7$ instances for the subtour-relaxed BQP, respectively. However, when solving the subtour-relaxed R-BQP, Gurobi successfully solved all $40$ instances for both configurations, requiring an average computation time of merely $22.4$ seconds and $2.34$ seconds, respectively.  These results show that the reduction step is inexpensive relative to the downstream optimization effort while dramatically improving solvability.

\subsection{Numerical Results}
We compare two distinct approaches for solving the subtour-relaxed R-BQP on the reduced graph $G$: our SABB method and Gurobi.

Our complete dataset consists of 1,160 instances. Across all reduced graphs $G$, the number of vertices $|V|$ ranges from 4 to 473, and the number of arcs $|E|$ ranges from 4 to 1,293. In addition, the sparsity of the reduced quadratic coefficient matrices ranges from 0.125 to 0.870. 

To systematically evaluate scalability, we partition the instances into bins based on their arc count, $|E|$.  We use $|E|$ as the primary measure for problem complexity because it directly corresponds to the number of binary decision variables. We randomly select 40 instances from each bin to form our test set. To ensure a fair comparison, both Gurobi and SABB were allocated a time limit of one hour per instance.\footnote{Note that for instances where the 3600-second time limit is reached, the recorded total computation time may slightly exceed 3600 seconds. This occurs because our SABB performs the termination check only after the completion of a node evaluation.} Table \ref{tab:optgap_binned} reports the following performance metrics: 
\begin{itemize}
    \item \textbf{Avg. Gap:} The average relative optimality gap\footnote{The relative optimality gap is calculated as $\frac{|f_{UB} - f_{LB}|}{\max(|f_{UB}|, 10^{-8})}$, where $f_{UB}$ and $f_{LB}$ denote the best upper and lower bounds obtained at termination, respectively.} at termination.
    \item \textbf{Solved to Opt.:} The number of instances solved to optimality within the time limit.
    \item \textbf{Avg. Time (s):} The average computation time in seconds.
\end{itemize}
 
\begin{table}[htbp]
\caption{Comparison of optimality gaps, solvability, and average runtimes by arc count $|E|$ for SABB and Gurobi.}
\label{tab:optgap_binned}
\begin{tabular*}{\textwidth}{@{\extracolsep\fill}l  cc c c c c}
\toprule
\textbf{Arc bin ($|E|$)}  & \multicolumn{2}{c}{\textbf{Avg. Gap}} & \multicolumn{2}{c}{\textbf{Solved to Opt.}} & \multicolumn{2}{c}{\textbf{Avg. Time (s)}}\\
\cmidrule(lr){2-3} \cmidrule(lr){4-5}  \cmidrule(lr){6-7}
&   SABB & Gurobi & SABB & Gurobi & SABB & Gurobi \\
\midrule
$[1,100)$   & 0.000 & 0.000 & 40 & 40 & 0.32 & 0.15 \\
$[100,200)$ & 0.000 & 0.005 & 40 & 39 & 13.99 & 159.56 \\
$[200,300)$ & 0.000 & 0.080 &  40 &  33 & 96.08 & 1153.26 \\
$[300,400)$ & 0.139 & 0.667 &  24 &   8 & 1794.54 & 3121.43 \\
$[400,500)$ & 0.067 & 0.739 &  29 &   2 & 1924.37 & 3502.20\\
$[500,600)$ & 0.603 & 1.244 &   11 &   4 & 2915.98 & 3272.80 \\
$[600,700)$ & 0.957 & 1.795 &   2 &   2 & 3495.50 & 3498.32  \\
$[700,800)$ & 0.783 & 1.963 &   0 &   0 & 3611.74 & 3600.25  \\
$\geq 800$  & 0.972 & 2.272 &   1 &   0 & 3653.39 & 3600.18 \\
\bottomrule
\end{tabular*}
\end{table}
Table \ref{tab:optgap_binned} summarizes the computational results. For small instances ($|E| < 100$), both solvers solve all instances to optimality, with Gurobi requiring slightly less time. However, as the problem size increases to the medium-sized bins ($100 \le |E| < 300$), SABB demonstrates superior scalability. It continues to solve all instances to optimality, whereas Gurobi frequently reaches the one-hour time limit. Furthermore, SABB requires significantly less computation time.

As the number of arcs increases ($300 \le |E| < 600$), Gurobi’s success rate drops sharply, while SABB remains highly effective. The performance difference is most evident in the $[400,500)$ bin, where SABB successfully solves 29 out of 40 instances to optimality, while Gurobi only solves 2 instances. For the most challenging large-scale instances ($|E| \ge 700$), Gurobi completely fails to solve any instance to optimality within the one-hour time limit. In contrast, SABB demonstrates superior robustness, still managing to solve one instance to optimality even in the $|E| \ge 800$ bin. 

In terms of solution quality, as the problem size increases, SABB consistently yields a significantly smaller average optimality gap compared to Gurobi across all test instances. On the largest instances ($|E| \ge 800$), Gurobi's average optimality gap deteriorates to $2.272$ due to the poor quality of its lower bounds. In contrast, our SABB maintains an average gap of 0.972 by providing significantly tighter lower bounds produced by the SDP relaxation, indicating that it remains effective at bounding the objective value even when exact solutions are computationally demanding. 

In summary, these computational results establish SABB as a highly robust and scalable approach for solving the subtour-relaxed R-BQP. By generating significantly tighter lower bounds, SABB consistently speeds up convergence to optimality and maintains superior solution quality across all instance sizes, effectively overcoming the severe performance degradation observed in the commercial solver Gurobi.

\thispagestyle{plain}
\printbibliography  
\appendix
\section{Detailed Reduction Statistics}
\begin{table}[htbp]
\centering
\footnotesize
\caption{Average arc counts before ($k|\mathcal{E}|$) and after the reduction procedure ($|E|$), together with the average computational time in seconds required per configuration.}
\label{tab:combined_stats}
\begin{tabular*}{0.8\textwidth}{@{\extracolsep\fill}lcccc}
\toprule
Pairs & Grid size & \multicolumn{1}{c}{Initial arcs} & \multicolumn{1}{c}{Remaining arcs} & \multicolumn{1}{c}{Time} \\
$k$ & $m_v$ & \multicolumn{1}{c}{$k|\mathcal{E}|$} & \multicolumn{1}{c}{$|E|$} & \multicolumn{1}{c}{(s)} \\
\midrule
2 & 20  & 100 & 55  & 0.05 \\
  & 40  & 242 & 193 & 0.30 \\
  & 60  & 389 & 351 & 2.73 \\
  & 80  & 538 & 521 & 3.72 \\
  & 100 & 690 & 674 & 1.08 \\
\midrule
3 & 20  & 133  & 35  & 0.09 \\
  & 40  & 342  & 168 & 0.38 \\
  & 60  & 563  & 406 & 2.37 \\
  & 80  & 785  & 713 & 11.50 \\
  & 100 & 1012 & 964 & 4.14 \\
\midrule
6 & 40  & 586  & 60  & 0.59 \\
  & 60  & 1009 & 180 & 1.66 \\
  & 80  & 1459 & 460 & 8.48 \\
  & 100 & 1906 & 857 & 58.04 \\
\midrule
7 & 40  & 641  & 44  & 0.65 \\
  & 60  & 1137 & 137 & 1.77 \\
  & 80  & 1656 & 329 & 5.64 \\
  & 100 & 2179 & 658 & 24.62 \\
\midrule
8 & 40  & 694  & 30  & 0.72 \\
  & 60  & 1251 & 104 & 1.90 \\
  & 80  & 1832 & 241 & 4.56 \\
  & 100 & 2438 & 492 & 16.23 \\
\midrule
10 & 60  & 1454 & 56 & 2.21 \\
   & 80  & 2181 & 146 & 4.83 \\
   & 100 & 2911 & 300 & 10.60 \\
\midrule
12 & 80  & 2472 & 92  & 5.39 \\
   & 100 & 3348 & 188 & 10.45 \\
\midrule
13 & 80  & 2602 & 87  & 5.72 \\
   & 100 & 3550 & 157 & 10.78 \\
\bottomrule
\end{tabular*}
\end{table}

\section{Derivation of the Facial Reduction Auxiliary System}
\label{app:aux_derivation}

To derive the auxiliary system used to obtain the exposing vector for the facially reduced SDP in Section~\ref{sec:sdp_relaxation}, we consider the Lagrangian of the feasibility problem for the facially reduced SDP. We seek to determine whether a strictly feasible $R \succ 0$ exists or if a dual certificate exists proving that all feasible $R$ are restricted to the boundary of the positive semidefinite cone. We define the primal feasibility problem as finding $R \in \mathbb{S}^{r}_+$ and $Y \in \mathbb{S}^{n+1}$ such that $Y = VRV^\top$, $\textnormal{arrow}(Y) = e_0$, and $\psi(Y) = 0$. 

To construct the dual, we introduce the dual variables $z \in \mathbb{R}^{n+1}$ and $y \in \mathbb{R}^{|\mathcal{C}|}$ for the affine constraints, $X \in \mathbb{S}^{n+1}$ to represent the lifted matrix $Y = VRV^\top$, and $W \in \mathbb{S}^{r}_+$ as the dual slack matrix associated with the constraint $R \succeq 0$. The Lagrangian is given by:
\begin{equation*}
\min_{X, z, y, W \succeq 0} \max_{Y, R} \mathcal{L}(X, z, y, W, Y, R),
\end{equation*}
where the Lagrangian $\mathcal{L}$ is defined as:
\begin{equation*}
\mathcal{L} = \langle X, Y - VRV^\top \rangle + \langle z, e_0 - \textnormal{arrow}(Y) \rangle + \langle y, \psi(Y) \rangle + \langle W, R \rangle.
\end{equation*}
Using the properties of the adjoint operators $\textnormal{arrow}^*$ and $\psi^*$, we rearrange the terms as follows:
\begin{equation*}
\min_{X, z, y, W \succeq 0} \max_{Y, R} \langle X - \textnormal{arrow}^*(z) + \psi^*(y), Y \rangle + \langle W - V^\top XV, R \rangle + \langle z, e_0 \rangle.
\end{equation*}
For the inner maximization over $Y$ and $R$ to be bounded, the coefficients must vanish, leading to the dual feasibility requirements:
\begin{equation*}
\begin{aligned}
X - \textnormal{arrow}^*(z) + \psi^*(y) &= 0, \\
V^\top XV - W &= 0.
\end{aligned}
\end{equation*}
Under these conditions, the dual problem reduces to minimizing the remaining term $\langle z, e_0 \rangle$. According to the Theorem of the alternative, if the primal satisfies Slater's condition, any exposing vector $W \succeq 0$ must satisfy $\langle z, e_0 \rangle \leq 0$. This yields the auxiliary system \eqref{aux_slater} used to certify the failure of Slater's condition numerically.

\section{Numerical Certificate for Example ~\ref{ex:slater_failure}}
\label{app:W_matrix}
An exposing vector $W \in \mathbb{S}_{+}^{7}$ identifying the failure of Slater's condition for Example~\ref{ex:slater_failure} is given by:
\begin{equation*}
\small
W = \begin{bmatrix}
0.2140 & 0 & 0 & 0 & 0 & 0 & 0 \\
0 & 0.2141 & 0 & 0 & 0 & 0 & 0 \\
0 & 0 & 0.0953 & 0.0262 & -0.0953 & -0.0262 & -0.0287 \\
0 & 0 & 0.0262 & 0.0722 & -0.0262 & -0.0722 & 0.0174 \\
0 & 0 & -0.0953 & -0.0262 & 0.0953 & 0.0262 & 0.0287 \\
0 & 0 & -0.0262 & -0.0722 & 0.0262 & 0.0722 & -0.0173 \\
0 & 0 & -0.0287 & 0.0174 & 0.0287 & -0.0173 & 0.2367
\end{bmatrix}.
\end{equation*}

\end{document}